\documentclass[12pt,twoside]{amsart}
\usepackage{amssymb,amsmath,amsthm, amscd, enumerate, mathrsfs, upgreek}
\usepackage{graphicx, hhline, tikz}
\usepackage[colorlinks=true,pagebackref,hyperindex]{hyperref}
\usepackage[all]{xy}
\usepackage{color}
\usepackage[backrefs, alphabetic, initials]{amsrefs}
\usepackage{appendix}
\usepackage{color} 
\usepackage{latexsym}
\usepackage[T1]{fontenc} 
\usepackage{fancyhdr}
\usepackage{enumerate}
\usepackage{tikz-cd}
\usepackage{array, longtable}
\usepackage{caption}
\usepackage{mathtools}
\usepackage{float}
\usepackage[thicklines]{cancel}
\newcolumntype{C}{>{$}c<{$}}
\newcolumntype{L}{>{$}l<{$}}

\title[Degrees of canonical Fano threefolds]{Degrees of non-Gorenstein canonical Fano threefolds with Picard number one}

\date{\today}
\author{Minyou Li}
\address{Sun Yat-Sen University, School of Mathematics, Guangzhou 510275, China}
\email{\href{limy236@mail2.sysu.edu.cn}{limy236@mail2.sysu.edu.cn}}

\DeclareMathOperator{\Sing}{Sing}

\DeclareMathOperator{\Cl}{Cl}

\newcommand{\qW}{\text{\rm q}_{\text{\rm W}}}
\newcommand{\qQ}{\text{\rm q}_{\mathbb{Q}}}

\newcommand\lcm{{\text{l.c.m.}}}

\newtheorem{thm}{Theorem}[section]
\newtheorem{lem}[thm]{Lemma}
\newtheorem{prop}[thm]{Proposition}

\newtheorem{cor}[thm]{Corollary}

\theoremstyle{definition}
\newtheorem{ex}[thm]{Example}
\newtheorem{defn}[thm]{Definition}
\newtheorem{rem}[thm]{Remark}

\newtheorem{case}{Case}
\newtheorem{subcase}{Subcase}
\numberwithin{subcase}{case}

\makeatletter
 
 \@addtoreset{equation}{section}
\makeatother

\begin{document}

\begin{abstract}
   We show that the optimal upper bound for the anticanonical degrees of non-Gorenstein $\mathbb{Q}$-factorial canonical Fano threefolds with Picard number one is 200/3.
\end{abstract}

\maketitle 
\tableofcontents

\section{Introduction}\label{sec1}

In this paper, we work over the field $\mathbb{C}$ and use basic notation from \cite{km98}. We assume that all varieties appearing in this paper are projective and normal. A variety $X$ is called a \textit{Fano variety} if its first Chern class $c_1(X)$ is ample. A variety $X$ is said to have \textit{terminal} (resp. \textit{canonical}, \textit{klt}, \textit{lc}) singularities if the discrepancy $a(E,X)>0$ (resp. $a(E,X)\ge 0$, $a(E,X)>-1$, $a(E,X)\ge -1$) for every exceptional divisor $E$ over $X$. An exceptional divisor $E$ over $X$ is called a \textit{crepant divisor} if $a(E,X)=0$. A variety $X$ is called \textit{Gorenstein} if the Weil divisor $K_X$ is Cartier. A variety $X$ is \textit{$\mathbb{Q}$-factorial} if for every Weil divisor $D$ on $X$, there exists a positive integer $r$ such that $rD$ is Cartier on $X$. All the singularities above appear naturally in the Minimal Model Program (MMP for short). 

The \textit{(anticanonical) degree} $c_1(X)^3=(-K_X)^3$ plays an important role in the study of a Fano threefold $X$. Hence, computing the optimal upper bound for anticanonical degrees under different conditions is vital. 

 Prokhorov \cite{pr05} showed that for a Gorenstein canonical Fano threefold $X$, $c_1(X)^3\le 72$, and the equality holds if and only if $X\cong \mathbb{P}(1,1,1,3)$ or $\mathbb{P}(1,1,4,6)$. He \cite{pr07} also showed that for a non-Gorenstein $\mathbb{Q}$-factorial terminal Fano threefold $X$ with Picard number one, $c_1(X)^3\le \frac{125}{2}$, and the equality holds if and only if $X\cong \mathbb{P}(1,1,1,2)$. Later, Chen Jiang, Haidong Liu, and Jie Liu \cite{jll25} considered the invariants $\qW$ and $\qQ$ (See Definition \ref{q}) and showed that for a $\mathbb{Q}$-factorial canonical Fano threefold $X$ with Picard number one, $c_1(X)^3\le 72$, and the equality holds if and only if $X\cong \mathbb{P}(1,1,1,3)$ or $\mathbb{P}(1,1,4,6)$. Note that $\mathbb{P}(1,1,1,3)$ and $\mathbb{P}(1,1,4,6)$ are both Gorenstein. Based on \cite{jll25}, Chen Jiang, Tianqi Zhang, and Yu Zou \cite{jzz25} recently showed that for a canonical weak Fano threefold $X$, $c_1(X)^3\le 72$.

In this paper, we obtain the optimal upper bound for anticanonical degrees of non-Gorenstein $\mathbb{Q}$-factorial canonical Fano threefolds with Picard number one. Here is the main theorem of this paper.

\begin{thm}[Corollary \ref{cor}]
    For a non-Gorenstein $\mathbb{Q}$-factorial canonical Fano threefold $X$ with Picard number one, we have $c_1(X)^3\le \frac{200}{3}$. Moreover, if the equality holds, we have $\qQ(X)=\qW(X)\in \{2,4,10\}$. Suppose $c_1(X)^3=200/3$,
    
    (1) if $\qQ(X)=2$, then either $\Sing(X)$ contains exactly one curve $C_2$ which is of type $A_1$ with $-K_X\cdot C_2=\frac{1}{3}$ or $\Sing(X)$ is of codimension 3;

    (2) if $\qQ(X)=4$, then $\Sing(X)$ has exactly one curve $C_2$ which is of type $A_1$ with $-K_X\cdot C_2=\frac{1}{3}$;

    (3) if $\qQ(X)=10$ and if there exists a curve $C\subset \Sing(X)$, then $C$ is of type $A_n$ for $n\in \{1,2,3,4,5\}$.
\end{thm}

The following example implies that the upper bound $\frac{200}{3}$ can be achieved.
\begin{ex}
    There exists a toric $\mathbb{Q}$-factorial canonical Fano threefold $X\cong \mathbb{P}(1,1,3,5)$ with Picard number one, such that $c_1(X)^3=\frac{200}{3}$ and $\qW(X)=10$. We see that the two singular points on $X$ are of the form $\frac{1}{3}(1,-1,1)$ and $\frac{1}{5}(1,1,-2)$ respectively, where the former one is terminal and the latter one is canonical by \cite{ms84}*{Theorem 2.3}.
\end{ex}

\section{Preliminaries}

Here are the definitions of some invariants considered in this paper.
\begin{defn}\label{q}
    Let $X$ be a canonical Fano variety. We can define two kinds of \textit{Fano indices} by
    $$\qW(X):=\max\{q|-K_X\sim qB,\qquad B\in \Cl(X)\};$$
    $$\qQ(X):=\max\{q|-K_X\sim_{\mathbb{Q}} qB,\qquad B\in \Cl(X)\}.$$
\end{defn}

\begin{defn}

Let $X$ be a canonical threefold. According to \cite{reid87}*{(10.2)(3)}, there exists a basket of points
$$B_X=\{(r_i,b_i)\in \mathbb{Z}^2\ |\ 1\le i\le k;\ 1<b_i\le \frac{r_i}{2};\ \text{g.c.d.}(b_i,r_i)=1\}$$
for $X$, where a pair $(r_i,b_i)$ corresponds to an orbifold point $Q_i$ of type $\frac{1}{r_i}(1,-1,b_i)$. We call it \textit{Reid's basket} (basket for short).

Denote by $\mathcal{R}_X$ the collection of $r_i$ appearing in $B_X$. Let $r_X$ be the smallest positive integer such that $r_XK_X$ is Cartier and we call it the \textit{Gorenstein index}. Then 
$$r_X=\lcm \{r_i|r_i\in \mathcal{R}_X\}.$$
    
\end{defn}

\begin{thm}[\cite{reid87}*{(10.3)}, \cite{jll25}*{Theorem 2.1}]\label{rr}
    Let $X$ be a canonical Fano threefold and $B_X$ the basket for $X$. Then
    \begin{equation}\label{e1}
        c_2(X)\cdot c_1(X)+\sum_{r_i\in \mathcal{R}_X}(r_i-\frac{1}{r_i})=24;
    \end{equation}
    \begin{equation}\label{e2}
        \frac{1}{2}c_1(X)^3+3-\sum_{(r_i,b_i)\in B_X}\frac{b_i(r_i-b_i)}{2r_i}=h^0(X,-K_X)\in \mathbb{Z}_{\ge 0}.
    \end{equation}
\end{thm}

\begin{proof}
    Direct corollary of \cite{jll25}*{Theorem 3.8} and \cite{jll25}*{Theorem 4.6}.
\end{proof}

Here is one of the main tools of this paper, which expresses relationships between some invariants of a canonical Fano threefold.

\begin{thm}[\cite{jll25}*{Theorem 4.2}]\label{in}
    Let X be a canonical Fano threefold. Let $r_X$ be the Gorenstein index of $X$. Let A be a Weil divisor such that $-K_X\equiv \qQ(X)A$. Take $J_A$ to be the smallest positive integer such that $J_AA$ is Cartier in codimension 2. Then 

    (1)\ $J_Ar_X(-K_X)^3/\qQ(X)^2$ is a positive integer;

    (2)\ $\qW(X)\mid r_X(-K_X)^3$;

    (3)\ If $\qQ(X)=\qW(X)$, $J_A\mid \qQ(X)$.
\end{thm}

As is well known, the terminal singularities lie in codimension three. Hence, the singularities in codimension two are canonical or worse. So it is useful to consider the curves lying in $\Sing(X)$ for a canonical threefold $X$.
\begin{defn}
    Let $X$ be a canonical threefold. Given an irreducible curve $C\subset \Sing(X)$, we say that $C\subset X$ is \textit{of type} $\mathsf{A}_n(\text{resp.} \  \mathsf{D}_m,\ \mathsf{E}_k)$ if at a general point of $C$, $X$ is analytically  isomorphic to $\mathbb{A}^1\times S_C$ where $S_C$ is a Du Val singularity of type $\mathsf{A}_n\ (\text{resp.} \ \mathsf{D}_m,\ \mathsf{E}_k)$, where $n\ge 1\ (\text{resp.} \ m\ge 4,\ 6\le k\le 8)$. See \cite{jaf92}*{Theorem 1.1} for the types of Du Val singularities.

    Let $e_C$ be 1 plus the number of exceptional curves on the minimal resolution of $S_C$. Let $g_C$ be the order of the local fundamental group of $S_C$ (see \cite{llm19}*{Example 5.2}). Let $j_C$ be the order of the Weil divisor class group of $S_C$ (see \cite{lip69}*{Section 24}). These invariants are listed as follows.
    \begin{table}[!ht]
        \centering
        \caption{}
        \label{type}
        \begin{tabular}{c|ccccc}
            Type of $C$ & $\mathsf{A}_n$ & $\mathsf{D}_m$ & $\mathsf{E}_6$ & $\mathsf{E}_7$ & $\mathsf{E}_8$ \\
            \hline
            $e_C$ & $n+1$ & $m+1$ & $7$ & $8$ & $9$ \\
            $g_C$ & $n+1$ & $4m-8$ & $24$ & $48$ & $120$ \\
            $j_C$ & $n+1$ & $4$ & $3$ & $2$ & $1$ \\
        \end{tabular}
    \end{table}
\end{defn}

For a Weil divisor $D$ on $X$ and a general point $p\in C$, we see that locally $D\sim \mathbb{A}^1\times D_0$ around $p=(0,p_0)\in \mathbb{A}^1\times S_C$, where $D_0$ is a Weil divisor on $S_C$. Then we define
$$c_C(D):=c_{p_0}(D_0).$$
For $c_p(D)$, where $p\in S$ is a Du Val singularity and $D$ is a Weil divisor on the surface $S$, see \cite{reid87}.

The other main tool of this paper is the following lemma, which is related to the Kawamata--Miyaoka type inequality \cites{liu19, ijl25, ll25, ll252, jll25}.

\begin{lem}[\cite{jll25}*{Corollary 4.7}]\label{km}
Let X be a $\mathbb{Q}$-factorial canonical Fano threefold with Picard number one. Let A be an ample Weil divisor generating $\Cl(X)/\sim_{\mathbb{Q}}$. Take $J_A$ to be the smallest positive integer such that $J_AA$ is Cartier in codimension 2. Let $J_A=p_1^{a_1}p_2^{a_2}...p_k^{a_k}$ be the prime factorization, where $p_i$ are distinct prime numbers. Then
\begin{align*}
        0 \le \sum_{i=1}^k(p_i^{a_i}-\frac{1}{p_i^{a_i}}) 
        &\le \sum_{C\subset 
        \Sing(X)}(e_C-\frac{1}{g_C})(r_Xc_1(X)\cdot C) \\
        &\le \sum_{C\subset \Sing(X)} (j_C-\frac{1}{j_C})(r_Xc_1(X)\cdot C) \\
        &\le \begin{cases}
        r_Xc_2(X)\cdot c_1(X)-\frac{5}{16}r_Xc_1(X)^3 &\mathrm{if}\ q\le 5\\
        r_Xc_2(X)\cdot c_1(X)-\frac{q^2+2q-4}{4q^2}r_Xc_1(X)^3 &\mathrm{if}\ q\ge 6
        \end{cases}\\
        &\le r_Xc_2(X)\cdot c_1(X)-\frac{r_X}{4}c_1(X)^3.
\end{align*}
As a consequence, we have
    \begin{equation}\label{e3}
        c_1(X)^3\le\begin{cases}
        \frac{16}{5}c_2(X)\cdot c_1(X)  &\mathrm{if} \ q\le 5;\\
        \frac{4q^2}{q^2+2q-4}c_2(X)\cdot c_1(X) &\mathrm{if}\ q\ge 6.
    \end{cases}
    \end{equation}
\end{lem}

\begin{proof}
    The second third inequality comes from the proof of \cite{jll25}*{Corollary 4.7}. The fourth inequality is obtained by combining \cite{jll25}*{Theorem 3.8} and \cite{jll25}*{Theorem 4.6}. The first, the third, and the last inequalities are obvious.
\end{proof}

The following lemmas are useful in this paper.

\begin{lem}\label{bg}
    Suppose that X is a canonical Fano variety with Picard number one. If $\qQ(X)=1$, then the tangent sheaf $\mathcal{T}_X$ is $\alpha$-stable for every movable class $\alpha\in \text{\rm N}_1(X)_{\mathbb{R}}$. Furthermore, let $\dim X=3$ and $\alpha=c_1(X)^2$. By $\mathbb{Q}$-Bogomolov--Gieseker inequality (see \cite{kmm94}*{Lemma 6.5}) we have
    $$c_1(X)^3\le 3c_2(X)\cdot c_1(X).$$
\end{lem}

\begin{proof}
    This is a direct consequence of \cite{ll25}*{Proposition 3.6} and \cite{ou23}*{Theorem 1.4}. As mentioned in \cite{jll25}*{Remark 5.7}, $-c_1(\mathcal{F})$ is nef for any proper subsheaf $\mathcal{F}$ of $\mathcal{T}_X$. So $\mathcal{T}_X$ is $\alpha$-stable for every movable class $\alpha\in \text{\rm N}_1(X)_{\mathbb{R}}$.
\end{proof}
Chen Jiang and Haidong Liu proved the Riemann--Roch formula for canonical Fano threefolds (\cite{jl25}*{Theorem 4.7}). As a consequence, we have the following useful lemma.

\begin{thm}[\cite{jl25}*{Theorem 4.9}]\label{ic}
    Let $X$ be a canonical Fano threefold and let $f\colon Y\rightarrow X$ be a sequential terminalization. Suppose that $-K_X\sim_{\mathbb{Q}}qA$ for some positive rational number $q$ and some ample Weil divisor $A$. Then for any integer $s$ with $0<s<q$,
    $$-\frac{1}{2}A^2\cdot K_X+\sum_{C\subset \Sing(X)}(-K_X\cdot C)c_C(A)-\sum_{(r,b)\in B_X}\frac{\overline{i_{f^{\lfloor *\rfloor}(A)}b}(r-\overline{i_{f^{\lfloor *\rfloor}(A)}b})}{2r}\in \mathbb{Z},$$
    where $i_{f^{\lfloor *\rfloor}(A)}$ is the local index of $f^{\lfloor *\rfloor }(A)$ at the orbifold point of type $(r,b)\in B_X$.
\end{thm}

\section{Computation of degrees}

Throughout this section, let $X$ be a $\mathbb{Q}$-factorial non-Gorenstein canonical Fano threefold with Picard number one. Let $q:=\qQ(X)$ and $\hat{q}:=\qW(X)$. Denote the ordinary Chern classes by $c_i:=c_i(X)$ for $i=1,2$. Let $A$ be an ample Weil divisor generating $\Cl(X)/\sim_{\mathbb{Q}}$ such that $-K_X\sim_{\mathbb{Q}}qA$. Take $J_A$ to be the smallest positive integer such that $J_AA$ is Cartier in codimension 2, and $J_A=p_1^{a_1}p_2^{a_2}...p_k^{a_k}$ be the prime factorization of $J_A$. This section is devoted to proving the following theorem.

\begin{thm}\label{com}
With the assumption above, suppose that $c_1^3>66$. Then all possibilities of $X$ are listed in Table \ref{0.0}, ordered by the anticanonical degrees and Fano indices.
\end{thm}

\begin{table}[H]
    \centering
    \caption{}
    \label{0.0}
    \begin{tabular}{llllll}
        $c_1^3$ & $B_X$ & $r_X$ & $c_2c_1$ & $q=\hat{q}$ & $J_A$ \\ 
        \hline
        336/5 & \{(5,2)\} & 5 & 96/5 & 84 & 21 \\ 
        200/3 & \{(3,1)\} & 3 & 64/3 & 40 & 8 \\ 
        200/3 & \{(3,1)\} & 3 & 64/3 & 20 & 20,10,4,2 \\ 
        200/3 & \{(3,1)\} & 3 & 64/3 & 10 & 5,2,1 \\ 
        200/3 & \{(3,1)\} & 3 & 64/3 & 5 & 1 \\
        200/3 & \{(3,1)\} & 3 & 64/3 & 4 & 2\\ 
        200/3 & \{(3,1)\} & 3 & 64/3 & 2 & 2,1 \\ 
        133/2 & \{(2,1)\} & 2 & 45/2 & 1 & 1 \\
    \end{tabular}
\end{table}

\begin{proof}
    By \cite{jll25}*{Theorem 1.1}, we assume that $c_1^3<72$. First, we claim that $q=\hat{q}$. Indeed, if $q\neq \hat{q}$, we have $q/\hat{q}\ge 2$ and there exists a torsion divisor. Then by \cite{km98}*{Definition 2.52}, we have an index one cover $f\colon X'\rightarrow X$ where $X'$ is a canonical Fano threefold and $(-K_{X'})^3\ge 2(-K_X)^3\ge \frac{400}{3}$ (See also \cite{reid87}*{(3.5)}), which contradicts \cite{jzz25}*{Theorem 1.1}. The proof is divided into three cases, depending on the value of $q$. Each case has two subcases, depending on the relationship between $q$ and $J_A$.

\subsection{\boldmath{$q\le 5$}}
    After combining (\ref{e1}) and (\ref{e3}) we have
    $$\frac{165}{8}<\frac{5}{16}c_1^3\le c_2c_1=24-\sum_{r_i\in \mathcal{R}_X}(r_i-\frac{1}{r_i}).$$
    Then $\sum_{r_i\in \mathcal{R}_X}(r_i-\frac{1}{r_i})<\frac{27}{8}=3.375$, hence $\mathcal{R}_X$ may be $\{2\}$, $\{3\}$ or $\{2,2\}$. we can therefore list  all possibilities of $r_Xc_1^3$ satisfying (\ref{e2}) and the above inequality in Table \ref{1.0}.

\begin{table}[H]
    \centering
    \caption{}
    \label{1.0}
    \begin{tabular}{lllll}
        $\mathcal{R}_X$ & $r_X$ & $r_Xc_2c_1$ & $r_Xc_1^3$ & $r_Xc_2c_1-\frac{5r_Xc_1^3}{16}$ \\ 
        \hline
        \{2\} & 2 & 45 & 133 & 3.4375\\
        \{2\} & 2 & 45 & 137 & 2.1875 \\
        \{2\} & 2 & 45 & 141 & 0.9375 \\
        \{3\} & 3 & 64 & 200 & 1.5\\
        \{2,2\} & 2 & 42 & 134 & 0.125 \\ 
    \end{tabular}
\end{table}   

\subsubsection{\boldmath{$J_A=q$}}
We have $q\mid r_Xc_1^3$ by Theorem \ref{in}(2). All possibilities from Table \ref{1.0} satisfying this property are listed in Table \ref{1.1}, where we cross out $q$ that do not satisfy Lemma \ref{km}.
\begin{table}[H]
    \centering
    \caption{}
    \label{1.1}
    \begin{tabular}{llllll}
        $\mathcal{R}_X$ & $r_X$ & $r_Xc_2c_1$ & $r_Xc_1^3$ & $q$ & $r_Xc_2c_1-\frac{5r_Xc_1^3}{16}$  \\ 
        \hline
        \{2\} & 2 & 45 & 133 & 1 & 3.4375\\
        \{2\} & 2 & 45 & 137 & 1 & 2.1875 \\
        \{2\} & 2 & 45 & 141 & 1,\cancel{3} & 0.9375 \\
        \{3\} & 3 & 64 & 200 & 1,2,\cancel{4},\cancel{5} & 1.5\\
        \{2,2\} & 2 & 42 & 134 & 1,\cancel{2} & 0.125 \\
    \end{tabular}
\end{table}
By Lemma \ref{bg}, we can also rule out the cases $q=1$ except the case $r_Xc_1^3=133$. Hence, we obtain two possibilities in Table \ref{0.0}.

\subsubsection{\boldmath{$J_A\neq q$}}
By Theorem \ref{in}(1)(3), $\frac{q}{J_A}$ is a square factor of $r_Xc_1^3$. Note that 200 is the only possibility of $r_Xc_1^3$ which has square factors in Table \ref{1.0}. By computation, $(q,J_A)$ must be one of the following possibilities in this case
$$(2,1),\ (4,2),\ (5,1).$$
So we obtain three possibilities in Table \ref{0.0}. 

\subsection{\boldmath{$q\ge 6$}}
Combining (\ref{e1}) and (\ref{e3}) we have 
$$\frac{33}{2}<\frac{1}{4}c_1^3<c_2c_1=24-\sum_{r_i\in \mathcal{R}_X}(r_i-\frac{1}{r_i}).$$
This implies $\sum_{r_i\in \mathcal{R}_X}(r_i-\frac{1}{r_i})<\frac{15}{2}=7.5$, then $\mathcal{R}_X$ is one of the following possibilities
$$\{2\},\{3\},\{4\},\{5\},\{6\},\{7\},\{2,2\},\{2,3\},\{2,4\},\{2,5\},\{2,6\},\{3,3\},\{3,4\},$$
$$\{3,5\},\{2,2,2\},\{2,2,3\},\{2,2,4\},\{2,3,3\},\{2,2,2,2\},\{2,2,2,3\}.$$
We can list in Table \ref{2.0} all possibilities of $r_Xc_1^3$ satisfying (\ref{e2}) and $66< c_1^3< min\{72,4c_2c_1\}$.
\begin{table}[H]
    \centering
    \caption{}
    \label{2.0}
    \begin{tabular}{lllll}
        $\mathcal{R}_X$ & $r_X$ & $r_Xc_2c_1$ & $r_Xc_1^3$ & $r_Xc_2c_1-\frac{r_Xc_1^3}{4}$  \\ 
        \hline
        \{2\} & 2 & 45 & 133,137,141 & $\le 11.75$ \\
        \{3\} & 3 & 64 & 200,206,212 & $\le 14$ \\
        \{4\} & 4 & 81 & 267,275,283 & $\le 14.25$ \\
        \{5\} & 5 & 96 & 334,336,344,346,354,356 & $\le 12.5$ \\
        \{6\} & 6 & 109 & 401,413,425 & $\le 8.75$ \\
        \{7\} & 7 & 120 & 468,472,474 & $\le 3$ \\
        \{2,2\} & 2 & 42 & 134,138,142 & $\le 8.5$ \\
        \{2,3\} & 6 & 119 & 403,415,427 & $\le 18.25$ \\
        \{2,4\} & 4 & 75 & 269,277,285 & $\le 7.75$ \\
        \{2,5\} & 10 & 177 & 673,677,693,697 & $\le 8.75$ \\
        \{2,6\} & 6 & 100 & None \\
        \{3,3\} & 3 & 56 & 202,208,214 & $\le 5.5$ \\
        \{3,4\} & 12 & 211 & 809,833 & $\le 8.75$  \\
        \{3,5\} & 15 & 248 & None & \\
        \{2,2,2\} & 2 & 39 & 135,139,143 & $\le 5.25$ \\
        \{2,2,3\} & 6 & 110 & 406,418,430 & $\le 8.5$ \\
        \{2,2,4\} & 4 & 69 & 271 & =1.25 \\
        \{2,3,3\} & 6 & 103 & 407 & =1.25 \\
        \{2,2,2,2\} & 2 & 36 & 136,140 & $\le 2$ \\
        \{2,2,2,3\} & 6 & 101 & 397 & =1.75 \\
    \end{tabular}
\end{table}

\subsubsection{\boldmath{$J_A=q$}} 
$q\mid r_Xc_1^3$ by Theorem \ref{in}(2). Hence we can list in Table \ref{2.1} all possibilities of $q\ge 6$ which satisfy $\sum_{i=1}^k(p_i^{a_i}-\frac{1}{p_i^{a_i}})<r_Xc_2(X)\cdot c_1(X)-\frac{r_X}{4}c_1(X)^3$ in Lemma \ref{km}. 

\begin{table}[H]
    \centering
    \caption{}
    \label{2.1}
    \begin{tabular}{llllll}
        $\mathcal{R}_X$ & $r_X$ & $r_Xc_2c_1$ & $r_Xc_1^3$ & $q$ & $r_xc_2c_1-\frac{r_Xc_1^3}{4}$ \\ 
        \hline
        \{2\} & 2 & 45 & $133=7\times 19$ & 7 & 11.75 \\
        \{3\} & 3 & 64 & $200=2^3\times 5^2$ & 8,10,20,40 & 14 \\
        \{4\} & 4 & 81 & $275=5^2\times 11$ & 11 & 12.25 \\
        \{5\} & 5 & 96 & $336=2^4\times 3\times 7$ & 6,7,8,12,14,21,24,28,42 & 12 \\
        \{5\} & 5 & 96 & $344=2^3\times 43$ & 8 & 10\\
        \{5\} & 5 & 96 & $354=2\times 3\times 59$ & 6 & 7.5 \\
        \{2,3\} & 6 & 119 & $403=13\times 31$ & 13 & 18.25\\
        \{2,3\} & 6 & 119 & $427=7\times 61$ & 7 & 12.25\\
        \{2,2\} & 2 & 42 & $138=2\times 3\times 23$ & 6 & 7.5 \\
        \{2,2,3\} & 6 & 110 & $406=2\times 7\times 29$ & 7,14 & 8.5
    \end{tabular}
\end{table}
But all possibilities in Table \ref{2.1} contradict $\sum_{i=1}^k(p_i^{a_i}-\frac{1}{p_i^{a_i}})\le r_Xc_2(X)\cdot c_1(X)-\frac{r_X(q^2+2q-4)}{4q^2}c_1(X)^3$ in Lemma \ref{km} except the case 
$$\mathcal{R}_X=\{3\},r_Xc_2c_1=64,r_Xc_1^3=200,q=20.$$

\subsubsection{\boldmath{$J_A\neq q$}}
We have
\begin{itemize}
    \item $\frac{q}{J_A}$ is a square factor of $r_Xc_1^3$;
    \item $J_A$ is a factor of $r_Xc_1^3/(\frac{q}{J_A})^2$ by Theorem \ref{in}(1)(3);
    \item $q\ge 6$ by assumption.
\end{itemize}
By the above properties, we can list in Table \ref{2.2} all possibilities for $q$ and $J_A$, where we cross out all cases that do not satisfy $\sum_{i=1}^k(p_i^{a_i}-\frac{1}{p_i^{a_i}})<r_Xc_2(X)\cdot c_1(X)-\frac{r_X}{4}c_1(X)^3$ in Lemma \ref{km}.

\begin{table}[H]
    \centering
    \caption{}
    \label{2.2}
    \begin{tabular}{lllllll}
        $\mathcal{R}_X$ & $r_X$ & $r_Xc_2c_1$ & $r_Xc_1^3$ & $\frac{q}{J_A}$ & $J_A$ & $r_xc_2c_1-\frac{r_Xc_1^3}{4}$ \\ 
        \hline
        \{3\} & 3 & 64 & $200=2^3\times 5^2$ & 2 & 5,10,\cancel{25},\cancel{50} & 14\\
        \{3\} & 3 & 64 & $200=2^3\times 5^2$ & 5 & 2,4,8 & 14\\
        \{3\} & 3 & 64 & $200=2^3\times 5^2$ & 10 & 1,2 & 14\\
        \{3\} & 3 & 64 & $212=2^2\times 53$ & 2 & \cancel{53} & 11\\
        \{4\} & 4 & 81 & $275=5^2\times 11$ & 5 & 11 & 12.25\\
        \{5\} & 5 & 96 & $336=2^4\times 3\times 7$ & 2 & 3,4,6,7,12,14,21,28,42,\cancel{84} & 12\\
        \{5\} & 5 & 96 & $336=2^4\times 3\times 7$ & 4 & 3,7,21 & 12 \\
        \{5\} & 5 & 96 & $344=2^3\times 43$ & 2 & \cancel{43},\cancel{86} & 10\\
        \{5\} & 5 & 96 & $356=2^2\times 89$ & 2 & \cancel{89} & 7\\
        \{6\} & 6 & 109 & $425=5^2\times 17$ & 5 & \cancel{17} & 2.75\\
        \{7\} & 7 & 120 & $468=2^2\times 3^2\times 13$ & 2 & 3,\cancel{9},\cancel{13},\cancel{39},\cancel{117} & 3 \\
        \{7\} & 7 & 120 & $468=2^2\times 3^2\times 13$ & 3 & 2,\cancel{4},\cancel{13},\cancel{26},\cancel{52} & 3 \\
        \{7\} & 7 & 120 & $468=2^2\times 3^2\times 13$ & 6 & 1,\cancel{13} & 3 \\
        \{7\} & 7 & 120 & $472=2^3\times 59$ & 2 & \cancel{59},\cancel{118} & 2 \\
        
    \end{tabular}
\end{table}

\begin{table}[H]
    \centering
    \renewcommand\thetable{7 continue}
    \caption{}
    \label{2.3}
    \begin{tabular}{lllllll}
        $\mathcal{R}_X$ & $r_X$ & $r_Xc_2c_1$ & $r_Xc_1^3$ & $\frac{q}{J_A}$ & $J_A$ & $r_xc_2c_1-\frac{r_Xc_1^3}{4}$ \\ 
        \hline
        \{2,5\} & 10 & 177 & $693=3^2\times 7\times 11$ & 3 & \cancel{7},\cancel{11},\cancel{77} & 3.75 \\
        \{3,3\} & 3 & 56 & $208=2^4\times 13$ & 2 & 4,\cancel{13},\cancel{26},\cancel{52} & 4 \\
        \{3,3\} & 3 & 56 & $208=2^4\times 13$ & 4 & \cancel{13} & 4 \\
        \{3,4\} & 12 & 211 & $833=7^2\times 17$ & 7 & 1,\cancel{17} & 2.75 \\
        \{2,2,2\} & 2 & 39 & $135=3^3\times 5$ & 3 & 3,5,\cancel{15} & 5.25 \\
        \{2,2,2,2\} & 2 & 36 & $136=2^3\times 17$ & 2 & \cancel{17},\cancel{34} & 2 \\
        \{2,2,2,2\} & 2 & 36 & $140=2^2\times 5\times 7$ & 2 & \cancel{5},\cancel{7},\cancel{35} & 1
    \end{tabular}
\end{table}
But all possibilities in Table \ref{2.2} contradict $\sum_{i=1}^k(p_i^{a_i}-\frac{1}{p_i^{a_i}})\le r_Xc_2(X)\cdot c_1(X)-\frac{r_X(q^2+2q-4)}{4q^2}c_1(X)^3$ in Lemma \ref{km} except the cases listed below.
$$r_Xc_1^3=200,q=\hat{q}=40,J_A=8;$$
$$r_Xc_1^3=200,q=\hat{q}=20,J_A=2,4,10;$$
$$r_Xc_1^3=200,q=\hat{q}=10,J_A=1,2,5;$$
$$r_Xc_1^3=336,q=\hat{q}=84,J_A=21.$$
Then we obtain the remaining eight possibilities in Table \ref{0.0}.
\end{proof}

\section{Reducing the cases}
\setcounter{case}{0}
In this section, we eliminate certain cases from Table \ref{0.0} and determine the optimal upper bound for anticanonical degrees. We still use the notation at the beginning of Section 3.

\begin{prop}\label{1}
    The case $c_1^3=336/5$ and the case $c_1^3=200/3, q=40$ in Table \ref{0.0} do not exist.
\end{prop}
\begin{proof}
    First, suppose $X$ satisfies $c_1^3=336/5$, $c_2c_1=96/5$, $q=\hat{q}=84$, $J_A=21$, $B_X=\{(5,2)\}$. By Lemma \ref{km}, we have
    \begin{align*}
        \sum_{C\subset \Sing(X)}(j_C-\frac{1}{j_C})(r_Xc_1(X)\cdot C) &\le \sum_{C\subset \Sing(X)}(e_C-\frac{1}{g_C})(r_Xc_1(X)\cdot C) \\
        &\le r_Xc_2c_1-\frac{84^2+2\times 84-4}{4\times 84^2}r_Xc_1^3 \\
        &= 10+\frac{1}{21}.
    \end{align*}
    By the definition of $J_A$ and $j_C$, we have $J_A\mid \lcm\{j_C\}$. By combining the above inequality and Table \ref{type}, we obtain that there are exactly two curves $C_3$ and $C_7$ in $\Sing(X)$, where $C_i$ is of type $\mathsf{A}_{i-1}$ for $i=3,7$. We also see that $c_1(X)\cdot C_3=c_1(X)\cdot C_7=\frac{1}{5}$. By Theorem \ref{ic} we have
    $$\frac{1}{210}-\frac{1}{5}\times \frac{2}{2\times 3}-\frac{1}{5}\times\frac{a(7-a)}{2\times 7}-\frac{b}{5}=-\frac{13+21a-3a^2+42b}{210}\in \mathbb{Z},$$
    where $a\in \{1,2,3\}$, $b\in \{0,2,3\}$. But this is impossible. Therefore, such an $X$ does not exist.

    Similarly, suppose $X$ satisfies $c_1^3=\frac{200}{3}$, $c_2c_1=64/3$, $q=\hat{q}=40$, $J_A=8$, $B_X=\{(3,1)\}$. By Lemma \ref{km}, we have
    \begin{align*}
        \sum_{C\subset \Sing(X)}(j_C-\frac{1}{j_C})(r_Xc_1(X)\cdot C) &\le \sum_{C\subset \Sing(X)}(e_C-\frac{1}{g_C})(r_Xc_1(X)\cdot C) \\
        &\le r_Xc_2c_1-\frac{40^2+2\times 40-4}{4\times 40^2}r_Xc_1^3 \\
        &= 11.625.
    \end{align*}
    By combining the above inequality and Table \ref{type}, we see that there must be a curve $C_8$ of type $A_7$ lying in $\Sing(X)$ with $c_1(X)\cdot C_8=\frac{1}{3}$. Moreover, $\Sing(X)$ may have curves which are of type $A_1$, $A_2$ or $A_3$. If there exists a curve $C_3$ of type $A_2$ lying in $\Sing(X)$, then $c_{C_3}(A)=0$ since $A$ is Cartier around a general point of $C_3$ by $J_A=8$. $c_1(X)$ has degree $\frac{1}{3}$ or $\frac{2}{3}$ on the curve which is of type $A_1$ and has degree $\frac{1}{3}$ on the curve which is of type $A_3$. By the above computation we note that for $a\in \{1,3\}$, $b\in \{0,1,2\}$, $c\in \{0,1,2\}$ and $d\in \{0,1\}$,
    \begin{align*}
        0>&\frac{1}{48}-\frac{1}{3}\times\frac{a(8-a)}{2\times 8}-\frac{b}{3}\times\frac{1}{2\times 2}-\frac{1}{3}\times \frac{c(4-c)}{2\times 4}-\frac{d}{3} \\
        \ge&\frac{1}{48}-\frac{1}{3}\times \frac{15}{2\times 8}-\frac{2}{3}\times\frac{1}{2\times 2}-\frac{1}{3}\times \frac{4}{2\times 4}-\frac{1}{3}  \\
        =&-\frac{23}{24}  \\
        >& -1.
    \end{align*}
    Hence, we have
    $$\frac{1}{48}-\frac{1}{3}\times\frac{a(8-a)}{2\times 8}-\frac{b}{3}\times\frac{1}{2\times 2}-\frac{1}{3}\times \frac{c(4-c)}{2\times 4}-\frac{d}{3}\notin \mathbb{Z},$$
    which contradicts Theorem \ref{ic}. 
\end{proof}

\begin{prop}\label{2}
    The case $c_1^3=200/3$, $q=20$ in Table \ref{0.0} does not exist.
\end{prop}

\begin{proof}
    Suppose $X$ satisfies $c_1^3=\frac{200}{3}$, $c_2c_1=64/3$, $q=\hat{q}=20$, $J_A=8$, $B_X=\{(3,1)\}$. We discuss the following two cases.

\begin{case}
    $\Sing(X)$ has no curve of type $A_1$. We have 
    \begin{align*}
        \sum_{C\subset \Sing(X)}(j_C-\frac{1}{j_C})(r_Xc_1(X)\cdot C) &\le \sum_{C\subset \Sing(X)}(e_C-\frac{1}{g_C})(r_Xc_1(X)\cdot C) \\
        &\le r_Xc_2c_1-\frac{20^2+2\times 20-4}{4\times 20^2}r_Xc_1^3 \\
        &= 9.5
    \end{align*}
    by Lemma \ref{km}. We use $C[a,b,c]$ to represent a curve $C$ of type $A_{a-1}$ where $b$ is the local index of $A$ at a general point $p\in C$ ($b=0$ for the case $A$ is Cartier at p) and $c=c_1(X)\cdot C$. Similar to the proof of Proposition \ref{1}, we can list all possibilities for curves lying in $\Sing(X)$ as follows. We omit the curves such that b=0 since they do not affect the computation of Theorem \ref{ic}. If there are two curves $C[a,b,c_1]$ and $C[a,b,c_2]$ in $\Sing(X)$, we regard them as one curve $C[a,b,c_1+c_2]$ for the same reason.  
    \begin{itemize}
        \item $J_A=20$: $\{C[4,1,\frac{1}{3}],C[5,1,\frac{1}{3}]\};\ \{C[4,1,\frac{1}{3}],C[5,2,\frac{1}{3}]\};$
        \item $J_A=10$: $\{C[4,2,\frac{1}{3}],C[5,1,\frac{1}{3}]\}; \ \{C[4,2,\frac{1}{3}],C[5,2,\frac{1}{3}]\};$
        \item $J_A=4$: $\{C[8,2,\frac{1}{3}]\};\ \{C[4,1,\frac{1}{3}]\};\ \{C[4,1,\frac{2}{3}]\};\ \{C[4,1,\frac{1}{3}],C[4,2,\frac{1}{3}]\};$
        \item $J_A=2$: $\{C[8,4,\frac{1}{3}]\};\ \{C[6,3,\frac{1}{3}]\};\ \{C[4,2,\frac{1}{3}]\};\ \{C[4,2,\frac{2}{3}]\}.$
    \end{itemize}
    By Theorem \ref{ic} we have
    \begin{equation}\label{4}
        \frac{1}{12}-\sum_{C[a,b,c]\in \Sing(X)}c\times \frac{b(a-b)}{2\times a}-\frac{d}{3}\in \mathbb{Z}
    \end{equation}
    for some $d\in \{0,1\}$. But every $C[a,b,c]$ listed above satisfies $\frac{1}{12}<c\times \frac{b(a-b)}{2\times a}\le \frac{1}{3}$. Also note that every case of $\Sing(X)$ listed above contains at most two curves and at least one curve. Hence, 
    \begin{align*}
        0>&\frac{1}{12}-\sum_{C[a,b,c]\in \Sing(X)}c\times \frac{b(a-b)}{2\times a}-\frac{d}{3} \\
        \ge &\frac{1}{12} - 2\times \frac{1}{3} -\frac{1}{3} \\
        >&-1.
    \end{align*}
    (\ref{4}) is impossible. There is a contradiction.
\end{case}

\begin{case}
    There exists a curve $C_2$ of type $A_1$ lying in $\Sing(X)$. By \cite{kaw15}*{Proposition 2.4}, there exists a crepant blow-up $f\colon Y\rightarrow X$ such that $f$ has exactly one exceptional prime divisor $E$ and $f(E)=C_2$. Since $j_{C_2}=2$, $B:=f^*(2A)$ is a well-defined Weil divisor in $Y$. Since $-K_X\sim 20A$ in $X$, we have $-K_Y=f^*(-K_X)\sim 10B$ in $Y$.

    Now we run a $K$-MMP on $Y$. Since $\rho(Y)=2$, we have following diagram
\[\begin{tikzcd}
	Y & {Y'} \\
	X & {X'}
	\arrow["g", dashed, from=1-1, to=1-2]
	\arrow["f"', from=1-1, to=2-1]
	\arrow["{f'}", from=1-2, to=2-2]
\end{tikzcd}\]
where $g$ is an isomorphism or the composition of finitely many flips, and $f'$ is a divisorial contraction or a Mori fibre space. Since the $K$-MMP of $Y$ ends up with a Mori fibre space, we have that either $X'$ is a $\mathbb{Q}$-factorial canonical Fano threefold with Picard number one or $X'$ is of dimension 1 or 2.
\begin{subcase}
    $f'$ is a Mori fibre space and $X'$ is of dimension 1 or 2. Let $B'$ be the strict transform of $B$ in $Y'$. Then the general fibre $F$ of $f'$ is a canonical del Pezzo surface or $\mathbb{P}^1$, and we have $-K_F=-K_{Y'}|_F\sim 10B'|_F$, where $B'$ is the strict transform of $B$ in $Y'$. Hence $\qW(F)\ge 10$, which contradicts \cite{wan24}*{Proposition 3.3}.
\end{subcase}

\begin{subcase}
    $f'$ is a divisorial contraction and $X'$ is a $\mathbb{Q}$-factorial canonical Fano threefold with Picard number one. Let $B''$ be the strict transform of $B$ in $X'$. We have $-K_{X'}\sim 10B''$, then $10\mid \qW(X')$. Note that in this case, we have
    $$-K_{Y'}+aE'=f'^*(-K_{X'})$$
    for prime divisor $E'$ and $a\in \mathbb{R}_{>0}$. Hence, we have 
    $$-K_{X'}^3=-K_{Y'}^3-a^2K_{Y'}\cdot E'^2 > -K_{Y'}^3=-K_X^3=\frac{200}{3},$$
    where the inequality is due to $K_{Y'}\cdot E'^2<0$ by Hodge index theorem. Therefore, $X'$ is Gorenstein by Theorem \ref{com} and Proposition \ref{1}. Then by \cite{jll25}*{Theorem 1.1} or \cite{pr05}*{Theorem 1.5}, $c_1(X')^3$ must be one of 68, 70, 72. By Theorem \ref{in} we have 
    $$10\mid \qW(X')\mid c_1(X')^3,$$
    which is possible only if $c_1(X')^3=70$. Note that $\qQ(X')\neq 70$ by \cite{jl25}*{Theorem 1.1}. Therefore we have $\qW(X')=\qQ(X')=10$. By Theorem \ref{in}(1), we have $J_{B''}=10$, where $J_{B''}$ is the smallest positive integer such that $J_{B''}B''$ is Cartier in codimension 2. But this implies
    $$2-\frac{1}{2}+5-\frac{1}{5}=6.3>3.7=24-\frac{10^2=2\times 10-4}{4\times 10^2}\times 70,$$
    which contradicts Lemma \ref{km}.    \qedhere
\end{subcase}
\end{case}
\end{proof}

\begin{rem}
    By using Theorem \ref{ic} and tedious computation as in Proposition \ref{1} and Proposition \ref{2}, we can also rule out the case that $c_1^3=200/3$, $q=10$, $J_A=5$.
\end{rem}

\begin{prop}\label{3}
    The case $c_1^3=200/3$, $q=5$ in Table \ref{0.0} does not exist.
\end{prop}

\begin{proof}
    Suppose $X$ satisfies $c_1^3=\frac{200}{3}$, $c_2c_1=64/3$, $q=\hat{q}=5$. Let $A$ be an ample Weil divisor generating $\Cl(X)/\sim_{\mathbb{Q}}$ such that $-K_X\sim_{\mathbb{Q}}qA$. Since $X$ does not satisfy $\mathbb{Q}$-Bogomolov-Gieseker inequality, we obtain that $\mathcal{T}_X$ is not $c_1$-semistable. As the argument in the proof of \cite{jll25}*{Theorem 3.8}, let $0=\mathcal{E}_0\subsetneq \mathcal{E}_1\subsetneq ...\subsetneq \mathcal{E}_l=\mathcal{T}_X$ be the HN filtration of $\mathcal{T}_X$, where $2\le l\le 3$. Denote by $r_i$ the rank of $\mathcal{F}_i:=(\mathcal{E}_i/\mathcal{E}_{i-1})^{**}$ and by $q_i$ the unique positive integer such that $c_1(\mathcal{F}_i)\equiv q_iA$. Then we have
    $$\sum_{i=1}^l r_i=3,\ \sum_{i=1}^l q_i=5,\ \frac{q_1}{r_1}>\frac{q_2}{r_2}>...>\frac{q_l}{r_l}>0.$$
    The only possibility is 
    $$l=2,\ q_1=2,\ r_1=1,\ q_2=3,\ r_2=2.$$ 
    Then by \cite{jll25}*{Lemma 3.1} and \cite{jll25}*{Theorem 4.6} we have 
    $$6c_2c_1-2c_1^3\ge -\frac{2\times (2-\frac{3}{2})^2}{5^2}c_1^3=-\frac{1}{50}c_1^3.$$
    Therefore,
    $$c_1^3\le \frac{100}{33}c_2c_1=\frac{6400}{99}\approx 64.65,$$
    which contradicts the assumption $c_1^3=200/3$.
\end{proof}

\begin{cor}\label{cor}
    For a non-Gorenstein $\mathbb{Q}$-factorial canonical Fano threefold $X$ with Picard number one, we have $c_1(X)^3\le \frac{200}{3}$. Moreover, if the equality holds, we have $\qQ(X)=\qW(X)\in \{2,4,10\}$. Suppose $c_1(X)^3=200/3$,
    
    (1) if $\qQ(X)=2$, then either $\Sing(X)$ contains exactly one curve $C_2$ which is of type $A_1$ with $-K_X\cdot C_2=\frac{1}{3}$ or $\Sing(X)$ is of codimension 3;

    (2) if $\qQ(X)=4$, then $\Sing(X)$ has exactly one curve $C_2$ which is of type $A_1$ with $-K_X\cdot C_2=\frac{1}{3}$;

    (3) if $\qQ(X)=10$ and if there exists a curve $C\subset \Sing(X)$, then $C$ is of type $A_n$ for $n\in \{1,2,3,4,5\}$.
\end{cor}

\begin{proof}
    By combining Theorem \ref{com}, Proposition \ref{1}, \ref{2} and \ref{3}, we obtain the first part of this corollary. When $\qQ(X)\in \{2,4\}$, we have 
    \begin{align*}
        \sum_{C\subset \Sing(X)}(j_C-\frac{1}{j_C})(r_Xc_1(X)\cdot C) &\le \sum_{C\subset \Sing(X)}(e_C-\frac{1}{g_C})(r_Xc_1(X)\cdot C) \\
        &\le r_Xc_2c_1-\frac{5}{16}r_Xc_1^3 \\
        &= 1.5
    \end{align*}
    by Lemma \ref{km}. Hence, either $\Sing(X)$ has exactly one curve $C_2$ which is of type $A_1$ with $-K_X\cdot C_2=\frac{1}{3}$ or $\Sing(X)$ is of codimension 3. If $\qQ(X)=4$, $J_A=2$, then $\Sing(X)$ should be the former case. 

    When $\qQ(X)=10$, we have
    \begin{align*}
        \sum_{C\subset \Sing(X)}(j_C-\frac{1}{j_C})(r_Xc_1(X)\cdot C) &\le \sum_{C\subset \Sing(X)}(e_C-\frac{1}{g_C})(r_Xc_1(X)\cdot C) \\
        &\le r_Xc_2c_1-\frac{10^2+2\times 10-4}{4\times 10^2}r_Xc_1^3 \\
        &= 6
    \end{align*}
    by Lemma \ref{km}. Hence, every curve $C\subset \Sing(X)$ is of type $A_n$ for $n\in \{1,2,3,4,5\}$.
\end{proof}

\section*{Acknowledgments}
The author would like to thank his supervisor Prof. Haidong Liu for helpful discussions and valuable suggestions. He would like to thank Prof. Chen Jiang and Prof. Jie Liu for helpful discussions. He would like to thank the referee for valuable suggestions and comments.


\end{document}